\documentclass{article}
\usepackage[colorlinks, citecolor=blue]{hyperref}
\hypersetup{linkcolor=blue, anchorcolor=blue, citecolor=blue }

\usepackage{amscd}
\usepackage{times}
\usepackage{latexsym}
\usepackage{amsmath}
\usepackage{amssymb}
\usepackage{amsthm}
\usepackage{dsfont}
\usepackage{mathrsfs}
\usepackage{epsfig}
\usepackage{graphicx,graphics}
\usepackage{subfigure}
\usepackage{pstricks}
\usepackage{pst-all}
\usepackage[body={15cm,23cm}, top=2cm]{geometry}
\usepackage{bm}
\usepackage{multirow}
\usepackage{verbatim}
\usepackage{xcolor}
\usepackage{enumerate}
\usepackage{lineno}
\usepackage{algorithm}
\usepackage{algorithmic}
\usepackage{cases}

\newtheorem{theorem}{Theorem}[section]

\newtheorem{proposition}[theorem]{Proposition}
\newtheorem{lemma}[theorem]{Lemma}
\newtheorem{example}[theorem]{Example}

\newrgbcolor{pink1}{1 0.71 0.77}
\newrgbcolor{Seashell}{1 0.9607 0.9333}
\newrgbcolor{Snow}{1 0.9804 0.9804}
\newrgbcolor{Pink4}{0.5451 0.3882 0.4235}
\newrgbcolor{Honeydew4}{0.5137 0.5451 0.5137}
\newrgbcolor{Honeydew1}{0.9411 1 0.9411}
\newrgbcolor{MistyRose4}{0.5451 0.4902 0.4823}
\newrgbcolor{MistyRose1}{1 0.8941 0.8823}
\newrgbcolor{MistyRose3}{0.8039 0.7176 0.7098}
\newcommand{\R}{\mathbf{R}}
\newcommand{\q}{(\mathrm{Q}_\omega)\,}
\newcommand{\qa}{(\mathcal{Q}_r)\,}

\newcommand{\w}{\omega}

\newcommand{\N}{\mathbb{N}}

\def \st {\mathrm{ s.t. }}
\def \supp {\mathrm{ supp}}

\begin{document}
\title { Every critical point of an $\ell_0$ composite minimization problem is a local minimizer}
\author{ Xue Feng  \qquad Chunlin Wu \footnote{Corresponding author. Email: wucl@nankai.edu.cn} \\
\small  \emph{School of Mathematical Sciences,
Nankai University, Tianjin, China, 300071}
}
\date{}

\maketitle

\emph{\textbf{Abstract}
Nowadays, $\ell_0$ optimization model has shown its superiority when pursuing sparsity in many areas.
For this nonconvex problem, most of the algorithms can only converge to one of its critical points.
In this paper, we consider a general $\ell_0$ regularized minimization problem, where the $\ell_0$ "norm" is composited with a continuous map.
Under some mild assumptions, we show that every critical point of this problem is a local minimizer, which improves the convergence results of existing algorithms.
Surprisingly, this conclusion does not hold for low rank minimization, a natural matrix extension of $\ell_0$ "norm" of a vector.
}


\section{Introduction}
The $\ell_0$ "norm" is to count the number of nonzero elements in a vector,
and is a good measure of sparsity.
By now, $\ell_0$ minimization is at a heart position for sparse reconstruction
and has been adopted in many fields such as signal processing, dictionary learning, compressive sensing, machine learning, classification, morphologic component analysis, subset selection, and so on\cite{bao2016image,Baraniuk2007Compressive,Dong2013An,Ji2008Bayesian,Miller1990Subset,shen2016wavelet,Zhang2013}.
In this paper, we consider the following $\ell_0$ composite regularization problem:
\begin{equation}
\label{model}
 \min_{x \in \R^N} f(x) := f_1(x) + \| g(x)\|_0 + \delta_X(x),
\end{equation}
where $\delta_X(x)$ is an indicator function with $X$ as the domain of $x$,
and the following assumptions hold:
\begin{itemize}
  \item $f_1$ is proper, continuous and convex;  $\mathsf{dom} f_1$ is open;
  \item $g$ is a continuous map from $\R^N$ to $\R^M$ with $\ker{g_i},i=1,\cdots,M$ convex and closed;
  \item $X$ is convex and closed.
\end{itemize}
These assumptions are trivial, and a lot of functions meet the requirements.
For instance, in many the inverse problems, we have $f_1(x) = \|A x - b \|^2$ where $A$ is a given matrix;
$g$ is the identity map or other linear maps;
$X$ could be the box constraints, convex cones or others.

The problem \eqref{model} is nonconvex and nonsmooth, and it is an NP-hard problem to obtain its minimizer.
For algorithms to solve \eqref{model}, although we expect them to find one of its local minimizers which show good sparse properties \cite{Niko13l0,Feng_2018},
most of them are theoretically guaranteed to only converge to a critical point.
Specifically,
when $g$ is an identity map, $f_1$ is a least square function and $X$ equals $\R^N$, the problem \eqref{model} reduces to be
\begin{equation}
\label{modelcs}
\min_{x \in \R^N} \frac{\alpha}{2} \| Ax - b \|^2 + \| x \|_0,
\end{equation}
a key minimization model in compressed sensing.
The first type of algorithms to solve \eqref{modelcs} are greedy methods like orthogonal matching pursuit (OMP) \cite{tropp2007signal} and its variant CoSaMP\cite{needell2009cosamp}, which are easy to implement.
The OMP method was originally proposed to solve an $\ell_0$ constrained problem which is equivalent to \eqref{modelcs} \cite{Zhang2016On,Nikolova2016Relationship}.
The second type of methods include penalty decomposition (PD) method \cite{Lu2013Sparse} and iterative hard thresholding method (IHT)\cite{blumensath2008iterative,lu2014iterativesss}, which globally or locally (with subsequence) converge to some local minimizers.
The third type of algorithms are the so-called descent methods such as forward-backward splitting method, proximal alternating linearized method, proximal block coordinate descent (BCD) method\cite{attouch2010proximal,attouch2013convergence,Bolte2014}.
They are proven convergent to critical points of \eqref{modelcs}.
As can be seen, even for this simplest case, there are still some algorithms shown convergent to only critical points.

When $g$ is a general map, the minimization model \eqref{model} has wide applications in image processing and machine learning.
For example, when $g$ is a gradient operator, the model can be applied to image smoothing, debluring and denoising.
Penalty and alternating minimization methods work quite well to solve it without convergence guarantee\cite{xu2011image,Xu2013Unnatural}.
A wavelet frame based image restoration problem where $g$ denotes a fast tensor product framelet decomposition, is solved by PD method in \cite{Zhang2013} and doubly augmented Lagrangian (MDAL) method in \cite{Dong2013An}.
Neither of these has convergence result on the outer iteration procedures.
When $g$ is a surjective linear operator and some other assumptions hold,
Bregman alternating direction method with multipliers(ADMM) in \cite{wang2014convergence} and proximal ADMM in \cite{li2015global} are proven very recently to be able to generate sequences converging to a critical point of \eqref{model}, by the powerful KL property.
So far, in this general case, the best theoretical result is the convergence to some critical points.

From the above discussion,
it is natural to ask whether these convergent algorithms other than IHT and PD method can find a local minimizer of \eqref{model}.
Thus, the relationship between the local minimizers and critical points of \eqref{model} is necessary to study.
However, the related results are scarce in the literature, except for some very special cases\cite{Feng2018,bao2016image}.

The rank of a matrix is always considered as a natural extension of  the $\ell_0$ ``norm'' of a vector.
Many theoretical results on robust sparse recovery are generalized to low-rank reconstruction which arises in many applications like system identification\cite{liu2009interior}, data mining and pattern recognition\cite{elden2019matrix}, low-dimensional embedding\cite{linial1995geometry} and matrix completion\cite{wang2015orthogonal}.
For example, as the $\ell_1$ norm of a vector is the well-known convex relaxation of $\ell_0$ ``norm'',
a heuristic idea to approximate the rank of matrices is to use the nuclear norm (the sum of singular values), which is the most successful tool recently.
The recovery guarantee of the rank minimization and its relaxation problems is provided under the rank restricted isometry property(RIP), an adaption of the RIP in vector case\cite{recht2010guaranteed,wang2015orthogonal}.
Meanwhile, lots of low rank minimization algorithms are also generalized from solvers of $\ell_0$ regularized problems.
For example, the orthogonal rank-one matrix pursuit method for low rank matrix completion in \cite{wang2015orthogonal} is based on the OMP method.
The authors in \cite{huang2018rank} used hard thresholding operation for the matrix singular values to solve rank minimization for image denoising.
\cite{lu2015penalty} proposed a corresponding PD method for general rank minimization.
In this context, it is therefore natural to ask the same question about the local minimizers and the critical points of rank minimization.

In this paper, we show in Section \ref{l0r} that every critical point of the $\ell_0$ composite model \eqref{model} is actually a local minimizer.
Surprisingly, this result does not necessarily hold for the rank minimization as shown in Section \ref{rankr}.

\section{Notations}

Denote $\N = \{ 1,2,\dots, M \}$.
For any $x \in \R^N$, we define
$$
\supp_g(x) := \{ i  \;:\; g_i(x) \neq 0 \}.
$$
Denote the set of local minimizers of $f$ as $L_f$:
$$
L_f := \{ x \in \R^N \,:\, x \text{ is a local minimizer of } f \}.
$$
A critical point $x$ of $f$ means $0 \in \partial f(x)$.

For any given $\w \subseteq \N $, we define the following problem:
\begin{equation}
\label{qw}
\q
\qquad \qquad	
\left\{
\begin{aligned}
\min_{x \in X} \;\; & f_1(x) ,\\
\st\;\;\;  &g_i(x) = 0, \;\; i \in \w^c.
\end{aligned}
\right.	
\end{equation}
We denote
\begin{equation}
\label{cw}
C_{\w} \,:=\, \{x \in \mathbf{R}^N \;:\; g_i(x)= 0, \forall\, i \in \omega^c\}.
\end{equation}
Then, the feasible domain of $(\mathrm{Q}_\omega)$ is $X \cap C_{\w}$ which is convex and closed.
Therefore, $\delta_{X \cap C_{\w}}(x)$ is a regular function(Theorem 6.9, p203; Exercise 8.14, p310, \cite{Rockafellar1998Variational}).
Besides, $(\mathrm{Q}_\omega)$ is convex and equivalent to the following unconstrained problem:
$$
\min_{x \in \R^N}  f_1(x) + \delta_{X \cap C_{\w}}(x).
$$


\section{Every critical point of $f$ is a local minimizer}
\label{l0r}

In this section,
we will show that every critical point of $f$ is a local minimizer.

%
%

\begin{theorem}
	\label{crlo}
	Suppose that $\bar{x}$ is a critical point of $f$.
	Then $\bar{x}$ is a local minimizer of $f$.
\end{theorem}

The proof is given later.

For better expression, we denote
\begin{equation}
f_2: = \| g(x)\|_0 + \delta_X(x).
\end{equation}

\begin{lemma} \label{regu}
	The function $f_2$ is lsc.
	Specially,
	if $\bar{x} \in X$, then there exists an open ball at $\bar{x}$, denoted by $B(\bar{x})$, such that $\forall\, x \in B(\bar{x})$, one of the following two cases holds:
	\begin{align}
	&\mbox{(a)} \quad  f_2(x)= f_2( \bar{x})
	\Longleftrightarrow \supp_g(x) = \supp_g(\bar{x}), x \in X \Longleftrightarrow
	x \in B(\bar{x}) \cap C_{\supp_g(\bar{x})} \cap X,  \label{sigmaa}  \\
	&\mbox{(b)} \quad  f_2(x) \geq f_2(\bar{x}) +1  ,  \label{sigmab}
	\end{align}
	where $C_{\supp_g(\bar{x})}$ is defined in \eqref{cw}.
\end{lemma}
\begin{proof}
	Since $\ell_0$ ``norm'' is lsc and $g$ is continuous, one has $\supp_g(x)\supseteq \supp_g(\bar{x})$ if $x\rightarrow \bar{x}$.
	Meanwhile, $x \in C_{\supp_g(\bar{x})}$ means $\supp_g(x)\subseteq \supp_g(\bar{x})$.
	Then, the conclusion is straightforward.
\end{proof}

We then characterize the local minimizers of $f$.
The combinatorial nature of $\ell_0$ "norm" makes minimizing $f$ become minimizing several convex subproblems $\q$.

\begin{theorem}\label{localequvalent}
The set of local minimizers of $f$ reads as
	$$
	L_f  = \bigcup_{\w \subseteq \N} \{ x \in \R^N \,:\, x \text{ solves } \q   \}.
	$$
\end{theorem}
\begin{proof}
Firstly, we show that
for any given $\w \subseteq \N$,
if $\bar{x}$ solves $\q$,
then $\bar{x}$ is a local minimizer of $f$.

According to Lemma \ref{regu},
we divide the neighborhood $B(\bar{x})$ of $\bar{x}$  into two disjoint subsets: $B(\bar{x}) = B_1 \cup B_1^c$ where $B_1 =  B(\bar{x}) \cap C_{\supp_g(\bar{x})} \cap X$.

		Take an arbitrary $x \in B_1$.
		Since $x \in C_{\supp_g(\bar{x})} \cap X$, $x$ is a feasible point of $\q$.
		As $\bar{x}$ solves $\q$, we have $f_1(\bar{x})  \leq f_1(x) $.
		Moreover, applying \eqref{sigmaa} gives $f_2(\bar{x}) = f_2(x)$.
		Hence, $\forall \, x \in B_1$,
		$		f(\bar{x}) = f_1(\bar{x}) + f_2(\bar{x})  \leq f(x).$
		
		Take an arbitrary $x \in B_1^c$.
		We have $f_2(x)\geq f_2(\bar{x})+1$ according to \eqref{sigmab}. Since $f_1(x)$ is continuous, there must exist a neighborhood $\mathcal{O}(\bar{x})$ of $\bar{x}$ such that $\forall\, x \in \mathcal{O}(\bar{x}), f_1(x) \geq f_1(\bar{x})-1$.
		Hence, $\forall\, x \in B_1^c \cap \mathcal{O}(\bar{x})$, we have $f(\bar{x}) \leq f(x)$ as well.
		
		Consequently, $\forall\, x \in  B(\bar{x}) \cap \mathcal{O}(\bar{x})$, we have $f(\bar{x}) \leq f(x)$, which means $\bar{x}$ is a local minimizer of $f$.

\vspace{0.5cm}

Secondly, we show that if $\bar{x}$ is a local minimizer of $f$, then $\bar{x}$ solves $(\mathrm{Q}_{\bar{\omega}}) $ with $\bar{\omega} := \supp_g( \bar{x})$.

		Since $\bar{x}$ is a local minimizer of $f$, $\bar{x}$ is also a local minimizer of the following constrained problem:
		\begin{equation}\label{pf}
		\min_{x \in X} f(x)\quad   \text{s.t. \quad} g_i(x) = 0, \forall\, i \in \bar{\omega}^c.
		\end{equation}
		The feasible domain of \eqref{pf} is $X \cap C_{\bar{\omega}}$.
		Thus, there exists a neighborhood $\mathcal{O}(\bar{x})$ of $\bar{x}$ such that $\forall\, x \in X \cap C_{\bar{\omega}} \cap \mathcal{O}(\bar{x}) $, $f(x) \geq f(\bar{x})$.

		According to Lemma \ref{regu},
		$	\forall\, x \in X \cap  C_{\bar{\omega}} \cap B(\bar{x}),  f_2(x) = f_2(\bar{x}).	$
		Thus,
		$$
		\forall\, x \in \mathcal{O}(\bar{x})  \cap X \cap C_{\bar{\omega}}  \cap B(\bar{x}), \qquad f_1(x)  \geq f_1(\bar{x}) .
		$$
		It follows that $\bar{x}$ is a local minimizer of $(\mathrm{Q}_{\bar{\w}})$, whose feasible domain is also $X \cap C_{\bar{\omega}}$. 	
		Since $(\mathrm{Q}_{\bar{\w}})$ is a convex problem, we have $\bar{x}$ solves $(\mathrm{Q}_{\bar{\w}})$.

\end{proof}

Clearly, there are at most $2^N$ different subproblems $\q$.
Thus, by enumerating all $\w$, we can find all of the local minimizers of $f$.
Next, we will see that this result also helps to give the subdifferential of $f_2$ which is the key	 to discuss the critical points of $f$.

\begin{lemma} \label{f2}
The function $f_2$ is regular. In particular, given $x \in \R^N$ with $\w := \supp_g(x)$, one has
$$
\partial f_2(x) = \partial \delta_{X \cap C_{\w}} (x).
$$
\end{lemma}
\begin{proof}
By the definition, the regular subdifferential of $f_2$ at $x$ is
  $$
\begin{aligned}
	\hat{\partial} f_2(x)  &=\{ x^* \in \R^N \,:\,  \liminf_{\underset{y \neq x}{y \rightarrow x} } \frac{1}{\| y-x \|}[ f_2(y) - f_2(x) - \langle x^*, y-x \rangle   ] \geq 0 \} \\
[\,\text{ Lemma } \ref{regu}\,] & = \{ x^* \in \R^N \,:\, \liminf_{\underset{y \neq x}{y \rightarrow x, f_2(y) = f_2(x) } } \frac{1}{\| y-x \|}[  - \langle x^*, y-x \rangle   ] \geq 0\} \\
[\, \eqref{sigmaa} \,]  & =  \{ x^* \in \R^N \,:\, \liminf_{\underset{y \neq x}{y \rightarrow x, y \in X \cap C_{\w} } } \frac{1}{\| y-x \|}[  - \langle x^*, y-x \rangle   ] \geq 0  \}\\
 & = \{ x^* \in \R^N \,:\,  \liminf_{\underset{y \neq x}{y \rightarrow x, } } \frac{1}{\| y-x \|}[ \delta_{X \cap C_{\w}}(y) -\delta_{X \cap C_{\w} }(x) - \langle x^*, y-x \rangle   ] \geq 0 \} \\
 & =  \hat{\partial} \delta_{X \cap C_{\w}} (x).
\end{aligned}
$$
The subdifferential of $f_2$ at $x$ is
  $$
\begin{aligned}
\partial f_2(x)
& =  \{ x^* \in \R^N \,:\, \exists\, y \rightarrow x, f_2(y) \rightarrow f_2(x), \hat{\partial}f_2(y) \ni y^*  \rightarrow x^* \}\\
[\,\text{ Lemma } \ref{regu}\,]
& =  \{ x^* \in \R^N \,:\, \exists\, y \rightarrow x, f_2(y) = f_2(x), \hat{\partial}f_2(y) \ni y^*  \rightarrow x^* \}\\
[\, \eqref{sigmaa} \,]
& =  \{ x^* \in \R^N \,:\, \exists\, y \rightarrow x, y \in X \cap C_{\w}, \hat{\partial}\delta_{X \cap C_{\w}}(y) \ni y^*  \rightarrow x^* \}\\
& =  \{ x^* \in \R^N \,:\, \exists\, y \rightarrow x, \delta_{X \cap C_{\w}}(y) \rightarrow \delta_{X \cap C_{\w}}(x), \hat{\partial} \delta_{X \cap C_{\w}}(y) \ni y^*  \rightarrow x^* \}\\
&= \partial \delta_{X \cap C_{\w}}(x).
\end{aligned}
$$
Since $X \cap C_{\w}$ is convex, we have $\partial \delta_{X \cap C_{\w}}(x) = \hat{\partial }\delta_{X \cap C_{\w}}(x)$(Proposition 8.12, p308, \cite{Rockafellar1998Variational}).
Thus,
one has $\partial f_2(x) = \hat{\partial } f_2(x)$.

For the horizon subdifferential of $f_2$ at $x$,
$$
\begin{aligned}
 \partial^{\infty} f_2(x)
 &= \{ x^* \in \R^N \,:\, \exists\, y \rightarrow x, f_2(y) \rightarrow f_2(x) , y^* \in \hat{\partial}f_2(y),  \lambda_n \searrow 0,    \lambda_n y^*  \rightarrow x^* \}\\
  &= \{ x^* \in \R^N \,:\, \exists\, y \rightarrow x, y\in X \cap C_{\w}, y^* \in \hat{\partial}\delta_{X \cap C_{\w}}(y),  \lambda_n \searrow 0,   \lambda_n y^*  \rightarrow x^* \}\\
 & =  \partial^{\infty} \delta_{X \cap C_{\w}}(x) \\
 & = \partial \delta_{X \cap C_{\w}}(x) ^{\infty} \\
 & = \partial f_2(x) ^{\infty}.
\end{aligned}
$$
The penultimate equation is due to the fact that $\delta_{X \cap C_{\w}}(x)$ is regular(Exercise 8.14, p310, \cite{Rockafellar1998Variational}).

Finally, as $f_2$ is lsc, we have $f_2$ is regular(Corollary 8.11, p307, \cite{Rockafellar1998Variational}).

\end{proof}

Since $f_1$ is not necessarily differential, the addition rule of subdifferential is hard to obtain.
However, the regularity of $\ell_0$ "norm" guarantees this conclusion here.

\begin{lemma} \label{add}
For any $x \in \mathsf{dom} f$,
$$
\partial f(x) = \partial f_1(x)     + \partial f_2(x).
$$
\end{lemma}
\begin{proof}
Since $f_1$ is convex and continuous with $\mathsf{dom} f_1$ open, according to Proposition 8.12 in \cite{Rockafellar1998Variational}, one has
$$
\partial f_1 (x)  = \hat{\partial} f_1(x),
\qquad
\partial^{\infty} f_1(x) = \{ x^* | 0 \geq \langle x^*, y- x \rangle \text{ for all } x \in \mathsf{dom} f_1\} = \{ 0 \}.
$$
Meanwhile, by Corollary 8.10 in \cite{Rockafellar1998Variational}, one has $\partial f_1(x) \neq \emptyset$.
Then, applying Proposition 8.12 again gives $\partial^{\infty} f_1(x) = \partial f_1(x)^{\infty} $.
Thus $f_1$ is regular at $x$.

Since $\partial^{\infty} f_1(x)  = \{ 0\}$,
the only combination of vectors $x^*_i \in \partial ^{\infty} f_i(x)$ with $x^*_1  + x^*_2 = 0$ is $x^*_1 = x^*_2 =0$.
Finally, as $f_2$ is regular at $x$, by Corollary 10.9 in \cite{Rockafellar1998Variational},
one has
$ \partial f(x)  =\partial f_1(x)  + \partial f_2(x) .$

\end{proof}

Now we are ready to prove Theorem \ref{crlo}.
\begin{proof}
Denote $\bar{\w} = \supp_g(\bar{x})$.
According to Lemma \ref{f2},
$0 \in \partial f(\bar{x}) = \partial f_1(\bar{x})  + \partial f_2(\bar{x}) =  \partial f_1(\bar{x})  +  \partial \delta_{ X \cap C_{\bar{\w}}} (\bar{x})$.
Similar to the proof of Lemma \ref{add}, we can obtain
$\partial [f_1(\bar{x}) + \delta_{X \cap  C_{\bar{\w}}}] (\bar{x}) = \partial f_1(\bar{x})  +  \partial \delta_{X \cap  C_{\bar{\w}}} (\bar{x}) \ni 0$.
Thus, $\bar{x}$ is a critical point of $(\mathrm{Q}_{\bar{\w}}) $. As $(\mathrm{Q}_{\bar{\w}}) $  is convex,
$\bar{x}$  solves $(\mathrm{Q}_{\bar{\w}}) $.
Finally, applying Theorem \ref{localequvalent} yields the result.
\end{proof}

\emph{Remark.}
A similar result to show that a critical point of $\ell_0$ minimization problem is a local minimizer has been given in Lemma 3.4 of \cite{bao2016image} and our previous work, Theorem 3.6 of \cite{Feng2018}.
However, the conclusion in \cite{bao2016image} only applies to a simple form of $f$, i.e., $f(x) = \sum_i \lambda_i |x_i|_0 + \phi(x)$ where $\lambda$ is a given vector, $\phi$ is convex and \emph{$C^1$};
in \cite{Feng2018}, the authors considered the special case with $g=\nabla$
and differentiable $f_1(x) = \| Ax - b \|^2$.
Both models in \cite{bao2016image}\cite{Feng2018} are convenient for subdifferential calculus. Our result here is not a trivial generalization of theirs.
\emph{Remark.}
Nonconvex optimization, especially $\ell_0$ minimization, is still a very active research topic now, which suggests that new improvements will be made in the coming years.
As long as their convergence to a critical point is given, we can claim that it is also a local minimizer.
This property may also help to design an efficient framework to reach its global minimizer in the future.

\section{The critical point of rank minimization}
\label{rankr}
The rank of a matrix is often used to measure the order, complexity, the dimension of a model, etc\cite{mfezal,fazel2004rank,liu2010robust}.
The rank function and the $\ell_0$ ``norm''  are both l.s.c. and piecewise-constant valued,
so rank minimization is usually considered as a natural extension of $\ell_0$ minimization.
Some results in the previous section apply to the rank minimization,
but, surprisingly, a critical point here is not necessarily a local minimizer.

We consider the following general model:
\begin{equation}
\label{model2}
 \min_{A \in \R^{M \times N}} F(A):= F_1(A) + \mathsf{rank}\, G(A) + \delta_{\mathcal{X}}(A),
\end{equation}
where $\mathsf{rank}\, G(A)$ is the rank of $G(A)$, and
\begin{itemize}
\item  $F_1$ is smooth with $\mathsf{dom} F_1$ open;
\item $G$ is a continuous map from $\R^{M \times N}$ to $\R^{M'  \times N'}$;
\item $\mathcal{X}\subseteq \R^{M \times N}$ is convex and closed.
\end{itemize}
The set of local minimizers of $F$ is denoted as $L_F$:
$$
L_F := \{ A \in \R^{M \times N} \,:\, A \text{ is a local minimizer of } F \}.
$$

Denote $\N' = \{ 1,2,\dots, \min(M',N') \}$.
Similarly,
for any given $r \in \N'$,
we define problem $\qa$ as follows
\begin{equation}
\label{qa}
\qa
\qquad \qquad	
\left\{
\begin{aligned}
\min_{A \in \mathcal{X}} \;\; & F_1(A) ,\\
\st\;\;\;  &\mathsf{rank}\, G(A) = r.
\end{aligned}
\right.	
\end{equation}
We also denote
\begin{equation}
\label{cr}
\mathcal{C}_{r} \,:=\, \{A \in \mathbf{R}^{M \times N} \;:\; \mathsf{rank}\, G(A) = r\},
\end{equation}
which is a closed smooth manifold.
The feasible domain of $\qa$ is $\mathcal{X} \cap \mathcal{C}_{r}$.
More details about the critical point of $\qa$ can be seen in \cite{helmke1995critical}.

The following lemma is from \cite{hossei}.
\begin{lemma} \label{space}
	(\,Theorem 2.1, \cite{hossei}\,)
Let $A$ have rank $r$, column space $\mathcal{U}$, and row space $\mathcal{V}$. Then,
$$
\partial \delta_{\mathcal{C}_{r} }(A) = \mathcal{U} ^{\bot} \otimes \mathcal{V}^{\bot}.
$$
\end{lemma}

\begin{proposition} \label{adda}
	Given $A \in \mathsf{dom} F$ with $r = \mathsf{rank}\, G(A)$, one has
	$$
	\partial F(A)  = \partial (\, F_1 +  \delta_{\mathcal{X} \cap \mathcal{C}_{r}}\,)(A).
	$$
	Moreover, $\{A \in \mathbf{R}^{M \times N} :  0 \in \partial F(A)\}=\{A \in \mathbf{R}^{M \times N}  : A \text{ is a critical point of } \qa, r \in \N' \}$.
\end{proposition}
\begin{proof}
	Since $F_1$ is smooth and $F_2(A):=\mathsf{rank}\, G(A) + \delta_{\mathcal{X}}(A)$ is finite,
	by Exercise 8.8c in \cite{Rockafellar1998Variational},
	we have $\partial F(A) = \nabla F_1(A)     + \partial F_2(A)$ and $\partial ( F_1 +  \delta_{\mathcal{X} \cap \mathcal{C}_{r}})(A) =\nabla  F_1(A) + \partial \delta_{\mathcal{X} \cap \mathcal{C}_{r}} (A)$.
	Similar to the proof of Lemma \ref{f2}, we have $\partial F_2(A) = \partial \delta_{\mathcal{X} \cap \mathcal{C}_{r} }(A)$ which yields the result.
\end{proof}

This result means a critical point of $F$ is equivalent to a critical point of some $\mathcal{Q}_r$.

\begin{theorem} \label{localequvalenta}
	For	the set of local minimizers of $F$,
	$$
	L_F  = \bigcup_{r \in \N'}   \{ A \in \R^{M \times N} \,:\, A \text{ is a local minimizer of } \mathcal{Q}_r \} ,
	$$
	where $\qa$ is defined in \eqref{qa}.
\end{theorem}

The proof of above theorem is similar to that of Theorem \ref{localequvalent}.
Now, we have that the local minimizer of $F$ is necessarily a local minimizer of some $\mathcal{Q}_r$.
However, the feasible domain $\mathcal{C}_{r}$ of $\qa$ is generally nonconvex.
For example, when $G$ is the identity map,
$\mathcal{C}_{r} = \{A \in \mathbf{R}^{M \times N} : \mathsf{rank}\, A = r\}$ is nonconvex.
Thus, a critical point of $\mathcal{Q}_r$ is not necessarily a local minimizer of $\mathcal{Q}_r$.
Therefore, combining proposition \ref{adda} and theorem \ref{localequvalenta} gives that a critical point of $F$ is not necessarily a local minimizer of $F$;
see Example \ref{noA}.

\begin{example}\label{noA}
We consider the following minimization problems:
$$
\min_{A \in \R^{2 \times 2}} \quad F(A) := \left\| A - \begin{bmatrix} 2&1\\ 1&2 \end{bmatrix}   \right\|^2_F + \mathsf{rank}\, A,
$$
and
\begin{equation*}
(\mathcal{Q}_1) \qquad
\begin{aligned}
\min_{A \in \R^{2 \times 2}} \;\; & F_1(A) := \left\| A - \begin{bmatrix} 2&1\\ 1&2 \end{bmatrix}   \right\|^2_F ,\\
\st\;\;\;  &\mathsf{rank}\, A = 1.
\end{aligned}
\end{equation*}

Denote $ \bar{A} = \begin{bmatrix} 0.5 &-0.5 \\ -0.5 & 0.5 \end{bmatrix} $ and $\mathcal{C}_1 = \{A: \mathsf{rank}\, A = 1 \}$.
Then the column space of $\bar{A}$ is $\{ k(1,-1)^T \}$ and its row space is $\{ k(1,-1) \}$.
Since $-\nabla F_1(\bar{A}) = \begin{bmatrix} 3&3 \\ 3&3 \end{bmatrix} = \left( \begin{matrix} 3 \\ 3 \end{matrix} \right) (1 \quad1) \in \partial \delta_{\mathcal{C}_1}(\bar{A})$ by Lemma \ref{space},
we have $0 \in \partial (\,F_1 +  \delta_{ \mathcal{C}_{1}}\,)(\bar{A})$,
meaning that $\bar{A}$ is a critical point of $(\mathcal{Q}_1)$.
By Proposition \ref{adda},
$\bar{A}$ is a critical point of $F$.

However, for all $\epsilon \rightarrow 0$ with $ \epsilon <1$,
we have $\mathsf{rank}\, \bar{A}_{\epsilon} =1 $
where  $\bar{A}_{\epsilon} = \begin{bmatrix} 0.5+\epsilon &-0.5 \\ -0.5 & 0.5+ \frac{-0.5\epsilon}{0.5+\epsilon} \end{bmatrix} $,
and
$$
F_1(\bar{A}_{\epsilon})- F_1(\bar{A})  = \frac{\epsilon^2( \epsilon^2 - 2\epsilon-1  ) }{(0.5+\epsilon)^2}<0 .
$$
Thus, $\bar{A}$ is not a local minimizer of $(\mathcal{Q}_1)$.
Since $\mathsf{rank}\, \bar{A} =1 $,
$\bar{A}$ is neither a local minimizer of $(\mathcal{Q}_0)$ or $(\mathcal{Q}_2)$ defined as \eqref{qa}.
By Theorem \ref{localequvalenta},
$\bar{A}$ is not a local minimizer of $F$.

\end{example}

\section{Conclusion}
In this paper, we mainly showed that every critical point of an $\ell_0$ composite minimization problem is a local minimizer,
a result that is not necessarily true for rank minimization models.

\section*{Acknowledgement}
This work was supported by the National Natural Science Foundation of China (NSFC) (Nos. XXX).

\bibliographystyle{abbrv}
\bibliography{refer}

\end{document}